\documentclass[12pt]{amsart}

\usepackage[T1]{fontenc}
\usepackage[utf8]{inputenc} %insieme al precedente consente di scrivere le lettere accentate
\usepackage[english]{babel}	%chapter -> capitolo

\usepackage{amsmath} %formule
\usepackage{amssymb} %altri simboli
\usepackage{mathtools} %estensione di amsmath
\usepackage{amsthm}	%teoremi, definizioni, ecc.
\usepackage{amsfonts} %simboli, come insiemi numerici
\usepackage{tikz-cd} %diagrammi e disegni
\usepackage{eufrak} %font aggiuntivo

\usepackage[all,cmtip]{xy}

\usepackage{enumerate}

\usepackage{algorithmic}
\usepackage[ruled]{algorithm}
\usepackage{caption}

\usepackage{ulem}

\title[Intersection cohomology and Severi varieties of quartic surfaces]
{Intersection cohomology and Severi varieties of quartic surfaces}

\theoremstyle{plain}
\newtheorem{theorem}{Theorem}[section]
\newtheorem{corollary}[theorem]{Corollary}

\theoremstyle{definition}

\newtheorem{remark}[theorem]{Remark}
\newtheorem{notations}[theorem]{Notations}

\DeclareMathOperator{\con}{Con}
\DeclareMathOperator{\Def}{Def}
\DeclareMathOperator{\sing}{Sing}

\newcommand{\vc}{\ensuremath{\mathcal{V}}}
\newcommand{\vcd}{\ensuremath{\mathcal{V}}_{\delta}}
\newcommand{\ic}{\ensuremath{\mathcal{I}}}
\newcommand{\oc}{\ensuremath{\mathcal{O}}}
\newcommand{\ac}{\ensuremath{\mathcal{A}}}
\newcommand{\fc}{\ensuremath{\mathcal{F}}}

\newcommand{\ec}{\ensuremath{\mathcal{E}}}
\newcommand{\nc}{\ensuremath{\mathcal{N}}}

\newcommand{\hc}{\ensuremath{\mathcal{H}}}
\newcommand{\rc}{\ensuremath{\mathcal{R}}}

\newcommand{\tc}{\ensuremath{\mathcal{T}}}
\newcommand{\mc}{\ensuremath{\mathcal{M}}}
\newcommand{\lc}{\ensuremath{\mathcal{L}}}
\newcommand{\xc}{\ensuremath{\mathcal{X}}}
\newcommand{\qc}{\ensuremath{\mathcal{Q}}}
\newcommand{\Sc}{\ensuremath{\mathcal{S}}}

\newcommand{\bP}{\mathbb{P}}
\newcommand{\bPd}{{\mathbb{P}^{\vee}}}
\newcommand{\vX}{X^{\vee}}
\newcommand{\bQ}{\mathbb{Q}}
\newcommand{\bC}{\mathbb{C}}

\begin{document}
	
	\title[Intersection cohomology and Severi varieties of quartic surfaces]
	{Intersection cohomology and Severi varieties of quartic surfaces}

	%    Information for first author
	\author{Davide Franco }
	\address{Universit\`a di Napoli
		\lq\lq Federico II\rq\rq, Dipartimento di Matematica e
		Applicazioni \lq\lq R. Caccioppoli\rq\rq, Via Cintia, 80125
		Napoli, Italy} \email{davide.franco@unina.it}
	%    \thanks will become a 1st page footnote.
	%\thanks{}
	
	%    Information for second author
	\author{Alessandra Sarti}
	\address{Universit\'e de Poitiers, Laboratoire de Math\'ematiques et Applications,  UMR 7348 du CNRS, 11 bd Marie et Pierre Curie, 86073 Poitiers Cedex 9, France} \email{alessandra.sarti@univ-poitiers.fr}
	
	\abstract  We give two explicit versions of the decomposition theorem  of Beilinson, Bernstein and Deligne applied to the universal family of quartic surfaces of $\mathbb{P}^3$. The starting point of our investigation is the remark that the nodes of a quartic surface impose independent conditions to the linear system $\mid \oc_{\bP^3}(4)\mid$. Although this property is known in literature, we provide a different argument more suited to our purposes. By a result of \cite{DGF}, the independence of the nodes implies in turn that each component of Severi's variety is smooth of the expected dimension
	and that the dual variety  is a divisor with normal crossings around Severi's variety. This allows us to study the complex $R\pi{_*}\mathbb{Q}_{\xc}$, the derived direct image of the constant sheaf over the universal family of quartic surfaces $ \mathcal{X} \stackrel{\pi}{\longrightarrow} \bP^{34}$, both in the open set parametrizing smooth and nodal quartics and in a tubular neighborhood of the variety of Kummer surfaces. We obtain in both cases an explicit decomposition and a formality result for the complex $R\pi{_*}\mathbb{Q}_{\xc}$.

	\bigskip\noindent {\it{Keywords}}:  Kummer surfaces, Nodes,  Severi varieties, Intersection cohomology,
	Decomposition Theorem,  Normal functions,

	\medskip\noindent {\it{MSC2010}}\,:  Primary 14B05; Secondary 14E15, 14F05,
	14F43, 14F45, 14M15, 32S20, 32S60, 58K15.
	
	\endabstract

	\maketitle

	\section{Introduction}
	
	\vskip3mm
	
	We consider the Veronese variety of $\bP^3$ imbedded in $\bP:=\bP^{34}$ by means of the linear system of quartic surfaces:
	$$X\subset \bP:=\mid \oc_{\bP^3}(4)\mid=\bP^{34}.$$ 
	We also denote by $X^{\vee}\subset \bP ^{\vee}$ the \emph{dual hypersurface}
	of $X$ and by $\mathcal{X}$  the \emph{universal hyperplane family} $$\mathcal{X}\subset X\times \bPd,$$ equipped with its natural projections:
	$$ X \stackrel{q}{\longleftarrow} \mathcal{X} \stackrel{\pi}{\longrightarrow} \bPd, \quad \quad \dim \xc = 36.
	$$
	
	Our main aim in this paper is to take a closer look at Severi's varieties $\vcd\subset \vX$ in relation to the behavior of the derived direct image $R\pi{_*}\mathbb{Q}_{\xc}$  ``around'' the dual hypersurface $\vX$.
	Specifically,
	it is well known that a quartic surface of $\bP ^3$ can have at most $16$ \emph{ordinary double points} (\textit{nodes} for short, compare with \cite[pag. 9]{Catanese} and \cite[Chapter V, Section 16]{BHPV}). Fix $1\leq \delta \leq 16$ and denote by   $\vc_{\delta}\subset X^{\vee}$ the locally closed subset parametrizing quartic surfaces with $\delta$ nodes. Assume $S=X\cap H\subset \bP^3$, $H\in \vcd$, is a quartic surface with $\delta$ nodes.  It is well known that any nodal quartic surface is \emph{unobstructed}, namely  that, whatever the number of singularities of $S$ is,
	
	$$
	{\text{\it the nodes of $S$ impose independent conditions to the linear system $\mid\oc_{\bP^3}(4)\mid$}}.
	$$ 
	\vskip2mm
	
	This fact is generally proved by resorting to deformation theory and showing that a nodal K3 surface is unobstructed (compare with \cite[Definition 34]{Catanese}). In this way one sees that the deformations of the surface have a submersion onto the local deformations of the singularities, so that one can obtain independent smoothings of all the singular points 
	(see for instance \cite{Wahl} and \cite[Section 1.3]{Catanese}).  Another  approach that goes back to an idea of Severi consists in proving directly that the nodes of a surface of degree $d$ of $\bP^3$ impose  independent conditions on adjoints of degree $\geq 2d - 5$ (we recall that a $k$-adjoint of $S$ is a surface of degree $k$ containing the nodes).  This is accomplished by means of a careful  study of the strict transform of the surface in the blowing up of $\bP^3$ at the nodes (compare with \cite[Section 2]{Nobile}).  In this paper we provide another proof of the unobstructedness property which mainly makes use of the smoothness of the universal deformation of a minimal resolution of a quartic nodal surface in the moduli space of quasi-polarized K3 surfaces combined with the universal property of Hilbert scheme (compare with Theorem \ref{indepcond}). We do this both in order to be reasonably self-contained  and also because our approach is most suited for the  local study of the derived direct image $R\pi{_*}\mathbb{Q}_{\xc}$.
	Furthermore, we believe that the techniques used in this work can be  extended to other examples of nodal hyperplane sections of Fano threefolds.
	
	Combining deformation theory with \cite[Theorem 3.3]{DGF}, the property above implies in turn that $\vcd$ is smooth of the expected dimension
	and that (cfr. Corollary \ref{normcross})
	\emph{the dual variety $\vX$ is a divisor with normal crossings around Severi's variety $\vcd$}.
	\vskip2mm
	As we said above, our next aim in this paper is to apply this result to understand   the behavior of the derived direct image $R\pi{_*}\mathbb{Q}_{\xc}$ in some neighborhood of Severi's varieties. The main ingredient is the  \emph{decomposition theorem  of Beilinson, Bernstein and Deligne}.
	By \cite[Sec. 2]{DeCMHodgeConj}, the decomposition theorem applied to $\pi$ provides a non-canonical decomposition
	\begin{equation}
		\label{decThmintr}  R\pi{_*}\mathbb{Q}_{\xc}\cong \bigoplus_{i\in \mathbb{Z}}\bigoplus _{j\in \mathbb{N}}IC(L_{ij})[-i-34)], \quad \text{in} \,\,\, D_c^b(\bPd),
	\end{equation}
	where $D_c^b(\bPd)$ denotes the derived category of $\bQ$-vector sheaves on $\bPd$ and  $IC(L_{ij})$ are the intersection cohomology complexes (see e.g. \cite{DeCMBAMS} for the definition). The local system $L_{ij}$ is supported on a suitable locally closed stratum of codimension  $j$ in $\bPd$ usually called \emph{support} of the decomposition. If $H \in U:=\bPd \backslash X^{\vee}$, then $\xc_H$ is smooth. Thus
	$\pi^{-1}(U) \rightarrow U$ is a smooth fibration and the sheaves  $R^{i}\pi{_*}\mathbb{Q}_{\xc}$ restrict to  local systems on  $U$, in the following denoted by $R^{i}\pi{_*}\mathbb{Q}_{\xc} \mid _U$. The general fibre of 
	$R^{2}\pi{_*}\mathbb{Q}_{\xc}\mid _U$ represents the intermediate cohomology of the (smooth) general fibre of $\pi$.
	
	Usually, the unique linear systems that are explicitely known are  those supported in $U$ (the \textit{general support}). In our case we have $$L_{i0}=R^{i}\pi{_*}\mathbb{Q}_{\xc} \mid _U$$ (compare with \cite[(2.5)]{DeCMHodgeConj}). 
	On the contrary, the supports appearing in the splitting (\ref{decThmintr}) and the local systems $L_{ij}$  are generally rather mysterious objects when $j\geq 1$. Such supports are contained in the dual hypersurface $\vX$.
	
	Our second main result is that \textit{all the supports other than the general one are disjoint from Severi's varieties $\vcd$:}
	
	\begin{theorem}\label{thmintro}
		Set $N\subset X^{\vee}$ be the closed set parametrizing quartic surfaces with at least one singular point that is not a node. Let $N^*:= \bPd  \backslash N$ be  the Zariski  open set parametrizing  smooth and nodal quartic surfaces.	For every $i$ such that $0\leq i \leq 4$, we have 
		$$IC(L_{i0})[-34]\mid_{N^*}=R^{i}\pi{_*}\mathbb{Q}_{\xc}\mid_{N^*},\quad \text{in} \quad D_c^b(N^*). $$	
		In particular, the decomposition theorem (\ref{decThmintr}) looks like
		$$R\pi{_*}\mathbb{Q}_{\xc}\mid_{N^*} \cong \bigoplus_{0\leq i \leq 4} R^i \pi{_*}\mathbb{Q}_{\xc}[-i]\mid_{N^*}$$
		in the open set $N^*$ .
	\end{theorem}
	We observe in passing that the result above says also that \textit{the derived direct image complex $R\pi{_*}\mathbb{Q}_{\xc}\mid_{N^*}$ is quasi isomorphic to the direct sum of its cohomology sheaves}. In other words, Theorem \ref{thmintro} can be seen as a \emph{formality theorem} for $R\pi{_*}\mathbb{Q}_{\xc}\mid_{N^*}$ (we recall that in the theory of  differential graded algebras a differential complex is called {\it formal} if it is quasi-isomorphic to its cohomology \cite[Definition 2]{Manetti}).
	
	It is well known that the perverse cohomology sheaves appearing in the decomposition theorem are semisimple (cfr. \cite[Theorem 1.6.1]{DeCMBAMS}).
	It is known that the perverse sheaves $IC(L_{i0})[-34]\mid_{N^*}=R^{i}\pi{_*}\mathbb{Q}_{\xc}\mid_{N^*}$ are simple in the Zariski open set $N^*$ (cfr. Corollary \ref{simplesummands}). Of course, there is no reason why they remain simple in a neighborhood of some Severi's variety $\vcd$. In the last section we take a closer look to the splitting ``around'' the variety $\vc:= \vc_{16}$  parametrizing Kummer quartics, providing the splitting in simple perverse sheaves of the restriction of $R\pi{_*}\mathbb{Q}_{\xc}$ to a tubular neighborhood of $\vc$.
	
	\vskip2mm
	
	\textbf{Acknowledgements. } 
	The authors thank Ciro Ciliberto for useful conversations.
	
	\section{Components of the Severi's variety}
	
	\vskip2mm

	In this paper, we will use notations very similar to those of  \cite[Section 2]{DeCMHodgeConj}. Specifically, we set
	$$X\subset \bP:=\mid \oc_{\bP^3}(4)\mid\cong\bP^{34}$$ 
	the Veronese variety of $\bP^3$ imbedded in $\bP$ by means of the linear system of quartic surfaces. We also denote by $X^{\vee}\subset \bP ^{\vee}$ the \emph{dual hypersurface}
	of $X$. Consider moreover the \emph{universal hyperplane family} $$\mathcal{X}\subset X\times \bPd,$$ equipped with its natural projections:
	$$ X \stackrel{q}{\longleftarrow} \mathcal{X} \stackrel{\pi}{\longrightarrow} \bPd, \quad \quad \dim \xc = 36.
	$$
	The quartic surfaces of  the complete linear system $\mid \oc_{\bP^3}(4)\mid$  are the  fibres of  $\pi$:
	$$\xc _H:= \pi^{-1}(H)=X\cap H, \quad \quad \dim \xc_H=2.$$
	Let $\con(X)\subset \mathcal{X}$ denotes the \emph{conormal variety} of $X$:
	$$\con(X):=\{(p,H)\in X\times \bPd: \,\, TX_p\subseteq H \}, \quad \quad  \dim \con(X) = 33,$$
	where $TX_p$ denotes the embedded tangent space to $X$ at $p$. We
	set $\pi_1: \con(X) \to \bPd$ the restriction of
	$\pi:\mathcal{X} \stackrel{}{\longrightarrow} \bPd$ to $\con(X)$.
	In particular, we have 
	$$X^{\vee}=\pi_1 (\con(X)).$$ 
	
	It is well known that a quartic surface of $\bP ^3$ can have at most $16$ \emph{ordinary double points} (\textit{nodes} for short, compare with \cite[pag. 9]{Catanese}). Fix $1\leq \delta \leq 16$ and denote by   $\vc_{\delta}\subset X^{\vee}$ the locally closed subset parametrizing quartic surfaces with $\delta$ nodes. Assume $S=X\cap H\subset \bP^3$, $H\in \vcd$, is a quartic surface with $\delta$ nodes.  Our main aim in this section is to prove that, whatever the number of singularities of $S$ is,
	\vskip1mm
	$$
	{\text{\it the nodes of $S$ impose independent conditions to the linear system $\mid\oc_{\bP^3}(4)\mid$}}.
	$$ 
	\vskip2mm
	\noindent
	As we explained in the introduction, this issue has already been addressed by several authors but our approach is different as it is mainly based on the smoothness of the universal deformation of a minimal resolution of a quartic nodal surface and   the universal property of Hilbert scheme. 
	Futhermore, by \cite[Theorem 3.3]{DGF}, the independence of  nodes implies that $\vc_{\delta}$ is smooth of the expected dimension.
	
	\begin{notations}\label{notations1} Fix $S=X\cap H\subset \bP^3$, $H\in \vcd$, and denote by $\widetilde S$ the minimal resolution of $S$. Then $\widetilde S$ is a \emph{quasi-polarized K3 surface} of degree $4$ (and genus $3$) \cite[Definition II 4.1]{Huy}. Set $D:= \Def (\widetilde S)$ the \emph{universal deformation } of $\widetilde S$ \cite[Corollary VI 2.7]{Huy}. By the local Torelli theorem \cite[Proposition VI 2.8]{Huy}, $D$ embeds isomorphically, via the period map, in the period domain for K3 surfaces  \cite[Section VI 1.1]{Huy} which is an open set in a quadric $Q\subset \bP ^{21}\cong \bP( H^2(\widetilde S, \bC))$. If $h\in H^2(\widetilde S, \bC)$ denotes the fundamental class of the curve of genus $3$ providing the quasi-polarization, then $$D ':=D \cap h^{\perp}$$ \emph{represents the deformation space of $\widetilde S$ in the moduli space of quasi-polarized K3 surfaces} of degree $4$  \cite[Section VI 2.4]{Huy}.
	\end{notations}
	
	\begin{theorem}
		\label{indepcond}
		Let $S\subset \bP^3$ be a quartic surface with $\delta$ nodes 
		$$S=X\cap H\subset \bP^3, \quad H\in \vcd.$$ 
		Then the set $\Delta:=\sing S$ of singular points of $S$ imposes independent conditions to the linear system $\mid\oc_{\bP^3}(4)\mid$:
		$$\dim H^0(\ic_{\bP ^3, \Delta }(4))= 35 -\delta.$$
	\end{theorem}
	\begin{proof}
		Let $T_{\vc _{\delta}}$ denotes the Zariski tangent space of $\vc _{\delta}$ at $H$. By deformation theory (compare with \cite[p. 43]{CS}), we have the  identification:
		$$T_{\vc _{\delta}}\cong \mid H^0(\ic_{\bP ^3, \Delta }(4))\mid. $$ 
		The statement amounts to prove that 
		$$\dim T_{\vc _{\delta}}= 34 -\delta.$$
		We remark in passing that the condition above implies that $\vc _{\delta}$, in a neighborhood of $S$, is the complete intersection of the smooth branches of $X^{\vee}$ (at $S$), that have independent tangent hyperplanes
		\cite[Theorem 3.3]{DGF}.
		
		As in Notations \ref{notations1}, denote by $\widetilde S$ the minimal resolution of $S$. If $p:\mc \to D'$ is the universal deformation of $\widetilde S$ in the moduli space of quasi-polarized K3 surfaces of degree $4$, and
		$\Sc \to \mc$ the line bundle corresponding to $h$ (cfr. Notations \ref{notations1}), then $p_*(\Sc)$ is a vector bundle of rank $4$ on $D'$ by the  Riemann-Roch theorem (compare with  \cite[Lemma 2.1 and Proof]{Huy}).
		Consider the projective bundle 
		$$\fc:=\bP(Fr(p_*(\Sc))),$$
		where $Fr(p_*(\Sc))$  denotes the \emph{frame bundle} of the  direct image $p_*(\Sc).$
		The bundle $\fc$ is a principal bundle with fibre the projective linear group $\bP \mathbb{GL} (4, \bC)$ on $D'$ representing the family of deformations of $\widetilde S$ equipped with a birational morphism to a quartic in $\bP ^3$.
		Indeed, a point in $\fc$ is a pair $(F,S')$, where $S'$ belongs to $D'$  and $F$ is a frame of the linear system
		$$\mid H^0(S',\Sc\mid_{S'}) \mid ,$$
		thus representing a map
		$$S'\rightarrow \bP^3,$$
		whose image is a  quartic in $\bP^3$.
		
		By letting $S$ to vary in $D'$, we get a morphism
		\begin{equation}\label{univfam1}
			\gamma:\mc\times _{D'} \fc\rightarrow \bP^3 \times \fc,	
		\end{equation}
		whose image can be thought of as the \emph{universal family of  quartics that are birational to a surface in $D'$}. 
		
		By the universal property of Hilbert scheme, the image of (\ref{univfam1}) can be recovered from the universal family of quartic surfaces by pull-back through an analytic morphism 
		\begin{equation}
			\label{rho}
			\rho: \,\, \fc \longrightarrow	\bP ^{\vee} 
		\end{equation}
		(we observe in passing that the domain and the target of $\rho$ have the same dimension). 
		Denote by $x\in \fc$ the point corresponding to the composition 
		$$\widetilde S \rightarrow S \hookrightarrow \bP ^3.$$
		
		We claim that the differential of $\rho$ at $x$ has maximal rank. In order to prove the claim, we argue by contradiction and assume
		$$   0 \not = v=(v_1, v_2)\in \ker d\rho _x \subset \tc_{\fc,x} \cong \tc _{\bP \mathbb{GL} (4, \bC), e} \times \tc _{D', {\widetilde S}}$$
		We proceed in two steps, according that $v_2$ vanishes or not.
		
		If $v_2=0$, then $v$ is vertical and is represented by an element of the Lie algebra of $\bP \mathbb{GL} (4, \bC) _e$ fixing the equation of $S$.   In particular, the nodes $p_1, \dots , p_{\delta} \in S$ are fixed. Correspondingly, we have a vector field in $\bP^3$ which vanishes at each node. By abuse of notations, we still denote by $v$ such a vector field. Consider the blow-up of $\bP$ at the nodes:
		
		$$\widetilde{\bP }^3 :=Bl_{\Delta} \bP^3, \quad \Delta:=\{p_1, \dots ,p_{\delta}\},  $$
		
		\[
		\xymatrix{
			E \ar[d] \ar[r]^j& \widetilde \bP ^3\ar[d]^f \\
			\Delta \ar[r]^i & \bP ^3. 
		}
		\]

		Consider the exact sequence describing the tangent bundle of $\widetilde{\bP }^3$ \cite[Lemma 15.4]{Fulton}
		$$0\rightarrow \tc _{\widetilde{\bP }^3} \rightarrow f^* \tc_{\bP^3} \rightarrow j_* \qc \rightarrow 0,$$
		where $\qc$ denotes the universal quotient bundle of the exceptional divisor $E$. Since the vector field $v$ vanishes  in $\Delta$ its image in  $j_* \qc$ vanishes as well. Hence $v\not=0$  lifts to a vector field of $\widetilde{\bP }^3 $ which restricts to a vector field of our K3 surface $\widetilde S\subset \widetilde{\bP }^3 $. This leads to a contradiction owing to $H^0(\widetilde S , \tc _{\widetilde S })=0$ \cite[p. 330]{Huy}.
		
		Assume now $0\not =v_2\in \tc _{D', {\widetilde S}}$ and consider the following commutative diagram
		
		\[
		\xymatrix{
			&\fc \ar[d]^{\tilde p} \ar[r]^{\rho}&  \bP ^{\vee} \\
			B\ar[r]^g \ar[ru]^{\widetilde g}&D'  &  
		}
		\]
		where $B$ is the unit disc in $\bC$ and $\tilde g$ is a germ passing through $x$, with tangent vector $v$ (hence $g$ is tangent to $v_2$). Of course we can assume $g$ injective. So $g$ is a deformation of $\widetilde S$, with infinitesimal deformation given by $v_2$, in the universal deformation space, preserving the quasi-polarization $h$. Correspondingly, we have a curve 
		$$\mathcal C:=\rho \circ \tilde g (B)\subset \bP ^{\vee}$$
		in the space of quartics representing a deformation of $S$. By resolving the restrition to $\mathcal C$ of the universal family of $\bP $, we find again a deformation of $\widetilde S$ parametrized by the curve $\mathcal C$. 
		Specifically, let $\mathcal H\subset \bP^3 \times \bP^{\vee}$ be the universal family of quartic surfaces of $\bP^3$ and let $\mathcal H _{\mathcal C}$ be the family of quartics parametrized by $\mathcal C$ via base change:
		\[
		\xymatrix{
			\mathcal H _{\mathcal C} \ar[d]\ar[r]& \mathcal H \ar[d]\\
			\mathcal C \ar[r] & \bP ^{\vee}
		}
		\] 
		Resolving the singularities of $\mathcal H _{\mathcal C}$ if necessary we get a deformation of $\widetilde S$ parametrized by $\mathcal C$. By the universal property of $D=\Def (\widetilde S)$, such a deformation is provided by an analytic map
		$s:\mathcal C\rightarrow D'$, thus we have
		\[
		\xymatrix{
			&\fc \ar[d]^{\tilde p} \ar[r]^{\rho}& \mathcal C \ar[ld]^{s} \\
			B\ar[r]^g \ar[ru]^{\widetilde g}&D'  &  
		}
		\]
		This diagram is commutative. Indeed,
		since any quartic surface determines the K3 surface obtained by blowing up its singularities  we find
		$$y=\rho(x')\quad  \Longrightarrow \quad \widetilde p (x')=s(y), \quad \forall y\in \mathcal C=\rho(\widetilde g (B)).$$
		Thus we have
		$$  s\circ \rho \circ \tilde g=g. $$
		We find a contradiction
		$$0\not = v_2 = dg(1)= d(s\circ \rho)(d \tilde g (1))=ds( d\rho(v))=ds(0)=0 $$
		and the claim is so proved.
		
		The morphism $\rho: \fc \rightarrow \bP $ has injective differential so it is a local isomorphism around $x$. Let $C_1, C_2, \dots , C_{\delta}$ denote the fundamental classes of the $(-2)$-curves obtained by blowing-up the nodes of $S$ and consider the  subvariety
		$$D _{\delta}:= D\cap \langle h, C_1, C_2, \dots , C_{\delta}\rangle ^{\perp}$$
		parametrizing the deformations of $\widetilde S$ which preserve such curves in the Picard group.
		Accordingly, via base change, we get a subvariety of $\fc$: 
		\[
		\xymatrix{
			\fc _{\delta} \ar[d] \ar[r]& \fc\ar[d]^{\widetilde p} \\
			D_{\delta} \ar[r] & D' 
		}
		\] 
		By  \cite[Section VI 2.4]{Huy} we have 
		$$\rho ^{-1}(\vc _{\delta})=\fc _{\delta}.$$
		As in Notations \ref{notations1}, consider the period domain of K3 surfaces $Q \subset \bP^{21}=\bP(H^2(\widetilde S; \bC)).$ The tangent hyperplane to $Q$ at $\widetilde S$ is the projectivization of the  orthogonal subspace to a cocycle $c^{2,0}$ generating $ H^{2,0}(\widetilde S; \bC)$.
		If the Zariski tangent space $T_{\vc _{\delta}}$ had larger dimension than expected then the same would happen for
		$$d\rho ^{-1}T_{\vc _{\delta}}=\tc _{\fc_{\delta}, x}\cong T_e\bP \mathbb{GL} (4, \bC)\times \langle c^{2,0}, h, C_1, C_2, \dots , C_{\delta}\rangle ^{\perp}.$$
		In view of the non-degeneracy of the intersection form, this would amount to a linear relation in $H^2(\widetilde S; \bC)$ between
		$c^{2,0}, h, C_1, C_2, \dots , C_{\delta}$. This is impossible.
	\end{proof}
	
	\vskip2mm
	
	By \cite[Theorem 3.3]{DGF}, the previous result implies:

	\begin{corollary}
		\label{normcross} Fix $1\leq \delta \leq 16$. The locally closed set $\vcd\subset \vX$, parametrizing quartic surfaces with $\delta$ nodes, is smooth of the expected dimension 
		$$\dim \vcd =34 -\delta.$$
		Assume  $H\in \vcd$, so that $$S:=X\cap H\subset \bP^3,$$ is a quartic surface with
		$\delta$ nodes. For every sufficiently
		small ball $B\subseteq \bPd$ containing $H$, the variety  $B\cap \vX$ is a divisor of $B$ with normal crossings
		
	\end{corollary}
     \noindent	(we observe that the variety $X$ appearing in
     \cite[Theorem 3.3]{DGF} has even dimension. Nonetheless, the argument applies also when
	the dimension is odd).
	
	In short, this result is achieved by combining two fundamental facts. First of all,  by \cite[Proposition 3.3]{N:K} the morphism $\pi_1: \con(X) \to \bPd$ is unramified over the nodal locus. This  in turn implies that
	$\pi_1$ provides an isomorphism of a suitable analytic neighborhood of  a pair $(p, S)\in \con(X)$ ($p$ is a node of $S$) with a branch of $X^{\vee}$ at $S$. The other relevant fact  \cite[p. 209]{Harris}
	is to observe that the embedded tangent space of such a branch is $p$, viewed as a hyperplane of $\mathbb P ^{\vee}$. So the different branches of $X^{\vee}$ meet transversally at any point of $ \vcd $ as soon as the nodes  impose independent conditions to the linear system.

	\section{The decomposition theorem}
	
	We keep the same notations of the previous sections. If $H \in U:=\bPd \backslash X^{\vee}$, then $\xc_H$ is smooth. Thus
	$\pi: \pi^{-1}(U) \rightarrow U$ is a smooth fibration and the sheaves  $R^{i}\pi{_*}\mathbb{Q}_{\xc}$ restrict to  local systems on  $U$, in the following denoted by $R^{i}\pi{_*}\mathbb{Q}_{\xc} \mid _U$. The general fibre of 
	$R^{2}\pi{_*}\mathbb{Q}_{\xc}\mid _U$ represents the intermediate cohomology of the (smooth) general fibre of $\pi$.
	
	By \cite[Section 2]{DeCMHodgeConj}, the decomposition theorem applied to $\pi$ provides a non-canonical decomposition
	\begin{equation}
		\label{decThm}  R\pi{_*}\mathbb{Q}_{\xc}\cong \bigoplus_{i\in \mathbb{Z}}\bigoplus _{j\in \mathbb{N}}IC(L_{ij})[-i-34)], \quad \text{in} \,\,\, D_c^b(\bPd),
	\end{equation}
	($D_c^b(\bPd)$ is the derived category of $\bQ$-vector sheaves on $\bPd$). $L_{ij}$ denotes a suitable local system on a stratum of codimension  $j$ in $\bPd$.  In particular, we have 
	$$L_{i0}=R^{i}\pi{_*}\mathbb{Q}_{\xc} \mid _U$$ 
	(compare with \cite[(2.5)]{DeCMHodgeConj}, where a slightly different notation is used).  The intersection cohomology complexes $IC(L_{ij})$ are \textit{semisimple perverse sheaves} \cite[Section 5]{Dimca2}.
	
	\begin{notations}
		We denote by $N\subset X^{\vee}$ the closed set parametrizing quartic surfaces with at least one singular point that is not a node. 	
	\end{notations}
	
	Our main aim is to determine, as explicitly as we can, the decomposition above in $N^*$, the open set parametrizing either smooth or nodal quartic surfaces. Our main result is the following.
	
	\begin{theorem}
		\label{decthm} In the Zariski open set $N^*\subset \bPd$, parametrizing smooth and nodal quartic surfaces,  we have ($0\leq i\leq 4$):
		$$IC(L_{i0})[-34]\mid_{N^*}=R^{i}\pi{_*}\mathbb{Q}_{\xc}\mid_{N^*},\,\,\, \quad \text{in} \quad D_c^b(N^*). $$	
	\end{theorem}
	\begin{proof}
		First of all, we consider the easiest cases i.e. $i\not =2$. We observe that, if 	$i\not =2$, then the sheaf
		$ R^i\pi{_*} \mathbb{Q} _{\xc} $ is a trivial local system on $N^*$. 
		This is obvious for the extremal cases $i=0,4$, because the fibre at $H$ of the sheaf $ R^i\pi{_*} \mathbb{Q} _{\xc} $  is $H^i(X\cap H, \mathbb Q )\cong \mathbb Q$, if $i=0, 4$, and the cohomology in top and lowest degree is invariant via monodromy.  As for the cases $i=1, 3$, this is consequence of the fact that any nodal quartic $S$ has the cohomology which embeds  in the cohomology of its minimal resolution, which is a K3 surface. Thus we have 
		$$H^1(S)=H^3(S)=0, \quad \text{if} \quad S=X\cap H, \,\,\, \forall H\in N^*,$$
		hence $$R^{1}\pi{_*}\mathbb{Q}_{\xc}\mid_{N^*}=R^{3}\pi{_*}\mathbb{Q}_{\xc}\mid_{N^*}= 0.$$
		
		On the other hand, the complex $IC(L_{i0})[-34]\mid_{N^*}$ is defined as the intermediate extension of $R^i\pi{_*} \mathbb{Q} _{\xc}\mid_U$ \cite[Section 5.2]{Dimca2}, which is a trivial local system as we have just observed. In the light of the smoothness of $N^*$, the complex 
		$IC(L_{30})[-34]\mid_{N^*}$ is a trivial local system as well. 
		We conclude $$IC(L_{i0})[-34]\mid_{N^*}=R^{i}\pi{_*}\mathbb{Q}_{\xc}\mid_{N^*}, \,\,\,\, \forall i\not=2.   $$

		We are left with the hardest case
		\begin{equation}\label{toprove}
			IC(L_{20})[-34]\mid_{N^*}=R^{2}\pi{_*}\mathbb{Q}_{\xc}\mid_{N^*}.	
		\end{equation}
		By Corollary  \ref{normcross}, we can assume that  the intersection $N^* \cap \vX$ is a divisor with normal crossings in the open set $N^*$. 
		
		We start with the following claim
		\begin{equation}
			\label{claim}	\hc^l (IC(L_{20})[-34]) _{H}=0,   \quad \forall H\in N^* \cap \vX, \,\,\, \forall l\geq 1
		\end{equation}
		where   $\hc^l (IC(L_{20})[-34]) _{H}$ denotes the $l$-th cohomology of the complex $IC(L_{20})[-34]$ at $H\in N^* \cap \vX$.
		
		Since the intersection $N^* \cap \vX$ is a divisor with normal crossings in the open set $N^*$, the local system 
		$L_{20}\mid_{U}=R^{2}\pi{_*}\mathbb{Q}_{\xc}\mid_{U}$ has a canonical extension to a vector bundle $\rc $
		on $N^*$ (see \cite{Deligne} and \cite{Schnell}).
		Further, in a suitable neighborhood of any hyperplane $H\in N^* \cap \vX$, the equation of
		$\vX$ has the form $t_1 \dots t_{\delta}=0$ ($\delta\leq 16$) and the local system
		$L_{20}\mid_{U}$ has monodromy operators
		$T_1, \dots , T_{\delta}$, with $T_{\rho}$ given by moving around the
		hyperplane $t_{\rho}=0$. If we denote by $N_{\rho}$ the \emph{logarithm of
			the monodromy operator} $T_{\rho}$, by \cite{CKS} and \cite[p. 322]{KP}
		the cohomology  $$\hc^l (IC(L_{20})[-34]) _{H}$$ of the intersection cohomology complex at $H\in N^* \cap \vX$
		can be computed as the $l$-th cohomology of the complex of
		finite-dimensional vector spaces
		\begin{equation}\label{complexB}
		B^p:=\bigoplus _{i_1<i_2< \dots < i_p}N_{i_1}N_{i_2}\dots N_{i_p}\rc_H, 
	\end{equation}
		with differential acting on the summands by the rule
		$$N_{i_1}\dots \hat{N}_{i_r}\dots N_{i_{p+1}}\rc_H \stackrel{(-1)^{r-1} N_{i_r}}{\longrightarrow} N_{i_1}\dots N_{i_r}\dots N_{i_{p+1}}\rc_H.$$
		Fix $H\in N^* \cap \vX$. Since $S:=H\cap X$ is nodal, the logarithm of
		the monodromy operator $N_{\rho}$ acts according to the
		\emph{Picard-Lefschetz formula}. Furthermore, as $S$ has
		$\delta$ ordinary double points, the vanishing spheres are disjoint
		to each other and we have
		$$N_{\alpha} N_{\beta}=0, \quad \text{for any} \quad \alpha\not= \beta.$$
		So the complex above is concentrated in degrees $0$ and $1$:
		$$\hc^l (IC(L_{20})[-34]) _{H}=0,   \quad \forall H\in N^* \cap \vX, \,\,\, \forall l\geq 2.$$
		In order to prove the claim (\ref{claim}), we have only to take care of the case $l=1$.

		The complex (\ref{complexB}) leads to the following exact sequence
		\begin{equation}\label{computeIC}
			0\to \hc^0 (IC(L_{20})[-34]) _{H}\to
			\rc_H \to  \bigoplus_{\rho=1}^{\delta} N_{\rho}\rc _H \to \hc^1 (IC(L_{20})[-34]) _{H} \to 0.
		\end{equation}
		
		Consider a hyperplane $H_t\in U$ very near to
		$H$, such that $S_t:=H_t\cap X$ is smooth.  Denote by $B_{\rho}$ a small
		ball around the $\rho$-th node of $S$, and   by $\sigma_{\rho}$ the Milnor sphere in $S_t \cap B_{\rho}$.
		
		By the Picard-Lefschetz formula, the map 
		$$\rc_H \rightarrow  \bigoplus_{\rho=1}^{\delta} N_{\rho}\rc _H $$
		of (\ref{computeIC}) coincide with
		\begin{equation}
			\label{PL}	H^2(S_t) \rightarrow  \bigoplus_{\rho=1}^{\delta} \sigma_{\rho}^{\vee}\cdot \bQ, \quad \psi \in H^2(S_t) \rightarrow  \bigoplus_{\rho=1}^{\delta} \langle \psi, \sigma_{\rho}\rangle\sigma_{\rho}^{\vee}.
		\end{equation}
		
		On the other hand, by excision, we have
		$$H^l(S, \cup _{\rho} (S\cap B_{\rho}))\cong  H^l(S_t, \cup _{\rho} (S_t\cap B_{\rho})).$$
		Combining the exact sequence
		$$ 0=  H^{1}( \cup _{\rho} (S\cap B_{\rho})) \rightarrow H^2(S, \cup _{\rho} (S\cap B_{\rho}))\rightarrow
		H^2(S)\rightarrow H^{2}( \cup _{\rho} (S\cap B_{\rho})) =0 ,$$
		where we have taken into account the conic nature of isolated singularities of a divisor \cite{Milnor}, with the following one
		$$ \dots \rightarrow H^{1}( \cup _{\rho} (S_t\cap B_{\rho})) \rightarrow H^2(S_t, \cup _{\rho} (S_t\cap B_{\rho}))\rightarrow
		H^2(S_t)\rightarrow H^{2}( \cup _{\rho} (S_t\cap B_{\rho}))\rightarrow ,$$
		we find the exact
		sequence
		\begin{equation}
			\label{computeIC2}  0 \to   H^{2}(S)\to
			H^{2}(S_t)\to  H^{2}( \cup _{\rho} (S_t\cap
			B_{\rho}))\cong\bigoplus_{\rho=1}^{\delta} \sigma_{\rho}^{\vee}\cdot \bQ\to 0
		\end{equation}
		(recall that $H^{3}(S)=0$ holds true for any nodal quartic $S$).
		
		Combining (\ref{computeIC}), (\ref{PL}) and
		and (\ref{computeIC2}), we find
		$$\hc^0 (IC(L_{20})[-34]) _{H}\cong H^{2}(S)\cong (R^{2}\pi_* \bQ_{\xc})_{H} $$
		and
		$$\hc^1 (IC(L_{20})[-34]) _{H}=0.
		$$
		The claim (\ref{claim}) follows.
		
		By the claim, the intersection cohomology complex $IC(L_{20})[-34]$ is indeed a sheaf in   $ D_c^b(N^*)$, hence (\ref{toprove}) is a consequence of the isomorphism $$\hc^0 (IC(L_{20})[-34]) _{H}\cong (R^{2}\pi_* \bQ_{\xc})_{H}, \quad \forall H\in N^*, $$ just proved.
	\end{proof}

	\begin{corollary}\label{simplesummands}
		In the derived category $D_c^b(N^*)$, the splitting in simple summands provided by the decomposition theorem  (\ref{decThm})  is 
		$$R\pi{_*}\mathbb{Q}_{\xc}\mid_{N^*}\cong \bigoplus_{0\leq i \leq 4}R^{i}\pi{_*}\mathbb{Q}_{\xc}[-i]\mid_{N^*}= \bQ _{N^*}\oplus R^{2}\pi{_*}\mathbb{Q}_{\xc}[-2]\mid_{N^*}\oplus \bQ _{N^*}[4].$$
	\end{corollary}
	\begin{proof}
		Combining (\ref{decThm}) with Theorem  \ref{decthm}	we find
		$$R\pi{_*}\mathbb{Q}_{\xc}\mid_{N^*}\cong \bigoplus_{0\leq i \leq 4}R^{i}\pi{_*}\mathbb{Q}_{\xc}[-i]\mid_{N^*} \oplus P,$$
		for a suitable constructible complex $P$. For any $H\in N^*$, put $S:=X\cap H\subset \bP^3$ the corresponding quartic surface. We have 
		$$H^i(S, \bQ)=\hc^i(R\pi{_*}\mathbb{Q}_{\xc})_H\cong \hc^i(\pi{_*}\mathbb{Q}_{\xc}[-i])_H \oplus \hc^i(P)_H=H^i(S, \bQ)\oplus \hc^i(P)_H.$$
		Hence 
		$$\hc^i(P)_H=0, \quad \forall i,\,\,\forall H\in N^*,$$
		meaning that $P$ vanishes in $D_c^b(N^*)$. In order to conclude, it suffices to observe that 
		$R^{2}\pi{_*}\mathbb{Q}_{\xc}\mid_{N^*}$ is simple in the category of perverse shaves of $N^*$. By \cite[Theorem 5.2.12]{Dimca2}, this is equivalent to the irreducibility of the local system $R^{2}\pi{_*}\mathbb{Q}_{\xc}[-2]\mid_{U}$, which is a classical result (compare  for instance with \cite[Section 3.2.3]{Voisin2}).
	\end{proof}

	\begin{remark}
		\label{supports}\begin{enumerate}
			\item 
			Our result shows that no support  other than the general one intersects the Zariski open set $N^*$ (so the supports are contained in $N$). In particular, \textit{all the supports other than the general one are disjoint from Severi's varieties $\vcd$, whatever the number of nodes}.	
			\item The resulting isomorphism
			$$R\pi{_*}\mathbb{Q}_{\xc}\mid_{N^*}\cong \bigoplus_{0\leq i \leq 4}R^{i}\pi{_*}\mathbb{Q}_{\xc}[-i]\mid_{N^*}$$
			can be stated by saying  that \textit{the derived direct image complex $R\pi{_*}\mathbb{Q}_{\xc}\mid_{N^*}$ is quasi isomorphic to the direct sum of its cohomology sheaves}. In other words, Corollary \ref{simplesummands} can be seen as a \emph{formality theorem} for the complex $R\pi{_*}\mathbb{Q}_{\xc}\mid_{N^*}$.
		\end{enumerate}
	\end{remark}
	
	\section{The moduli space of Kummer surfaces.}
	
	\vskip2mm
	In this section our aim is to determine the decomposition in simple summands of $R\pi{_*}\mathbb{Q}_{\xc}\mid_T$, in a tubular neighborhood $T$ of the Severi variety of \emph{Kummer quartics}.
	\begin{notations}
		Set $\vc:= \vc _{16}$ the Severi variety parametrizing Kummer quartics. By Corollary \ref{normcross}, there is a tubular neighborhood $T$ of $\vc$ such that $\vX \cap T$ is a divisor with normal crossings in $T$. Set $T^0:= T\backslash \vX \cap T$, the locus of smooth quartic surfaces that are  ``near'' to a Kummer. Fix $H\in \vc$ and consider a hyperplane $H_t\in T^0$ very near to
		$H$, so that $S_t:=H_t\cap X$ is smooth. Denote by $B_{\rho}$, $1\leq \rho \leq 16$, a small
		ball around the $\rho$-th node of $S$, and   by $\sigma_{\rho}\in H_2(S_t, \mathbb Q)$ the corresponding Milnor sphere in $S_t \cap B_{\rho}$. By the Picard-Lefschetz formula, the span
		$\langle \sigma_1^{\vee}, \dots \sigma_{16}^{\vee} \rangle\subset  H^2(S_t, \mathbb Q)$ of their Poincar\'e duals extend via monodromy to a sub-local system $\Sigma \subset R^{2}\pi{_*}\mathbb{Q}_{\xc} \mid _{T^0}$.
	\end{notations}
	
	\begin{theorem}
		With notations as above,  the orthogonal $\Sigma^{\perp}$ extends to a local system in $T$,  denoted by $\overline{\Sigma} ^{\perp}$, such that 
		$$\overline{\Sigma} ^{\perp} \mid _{\vc}= \hc^0 (IC(L_{20})[-34]) \mid _{\vc}.$$ Furthermore, the decomposition in simple  perverse sheaves of $R\pi{_*}\mathbb{Q}_{\xc}\mid_T$ is the following
		$$R\pi{_*}\mathbb{Q}_{\xc}\mid_T\cong \bQ_T\oplus \bQ_T[-2] \oplus \overline{ \Sigma}^{\perp}_{prim}[-2]\oplus \iota_*\Sigma[-2]\oplus \bQ_T[-4],$$
		where $\iota: T^0 \to T$ denotes the inclusion, and where $\overline{ \Sigma}^{\perp}_{prim}$ denotes the primitive part of $\overline{ \Sigma}^{\perp}$.
	\end{theorem}
	
	\begin{proof}
		We use the  notations of the proof of Theorem \ref{decthm}. In particular,  the local system 
		$R^{2}\pi{_*}\mathbb{Q}_{\xc}\mid_{T^0}$ has a canonical extension to a vector bundle $\rc $
		on $T$,
		in a suitable neighborhood of any hyperplane $H\in \vc$ we have  the  monodromy operators
		$T_1, \dots , T_{16}$ and their logarithms $N_1, \dots , N_{16}$.
		
		As $T\cap \vX$ is a divisor with normal
		crossings, if we fix a small neighborhood $B$ of $H$ in
		$\bPd$, then the local fundamental group $\pi_1(B\backslash(B\cap
		\vX), H)\cong \mathbb{Z}^{\delta}$ is independent of $H\in
		\vc$. So, the rank of the map
		$$\rc_H\to \bigoplus_{\rho} N_{\rho}\rc_H$$
		appearing in the sequence (\ref{computeIC}) is invariant and its kernel is a vector bundle on
		$T$. Furthermore, from the description of the canonical extension of $\rc$
		given e.g. in \cite[sec. 2]{Schnell}, one infers that such a kernel
		is the $\mathbb{Z}^{\delta}$-invariant part  of the local system
		$(R^{2}\pi_* \bQ_{\xc })_{ \mid T^0}\otimes \mathbb{C}$, thus it locally coincides with the claimed extension $\overline{\Sigma} ^{\perp}$ in view of the Picard-Lefscetz formula.
		Moreover, since the tubular neighborhood $T$ is homeomorphic to a
		fiber bundle on $\vc$, the long exact sequence of homotopy
		groups of $T^0$
		$$\dots \rightarrow \mathbb{Z}^{\delta}\rightarrow \pi_1(T^0, H) \rightarrow \pi_1(T, H)\cong \pi_1(\vc, H)  \rightarrow 0$$
		shows that  $\overline{\Sigma} ^{\perp}$ is in fact a local system on the whole $T$.
		The isomorphism  $$\overline{\Sigma} ^{\perp} \mid _{\vc}= \hc^0 (IC(L_{20})[-34]) \mid _{\vc}$$
		follows just taking into account the exact sequence (\ref{computeIC}). Furthermore, the local system  $\overline{\Sigma} ^{\perp}$ contains a trivial local system $ \bQ_T$ spanned by the polarization, and we have
		$$\overline{\Sigma} ^{\perp}\cong \bQ_T \oplus \overline{\Sigma} ^{\perp}_{prim},$$
		where $\overline{\Sigma} ^{\perp}_{prim}$ denotes the sub local system spanned by the primitive part of the cohomology.

		Comparing the statement with Corollary \ref{simplesummands}, since we have
		$$ R^{2}\pi{_*}\mathbb{Q}_{\xc}\mid_{T} \cong \bQ_T \oplus \overline{\Sigma} ^{\perp}_{prim}\oplus \iota_*\Sigma,$$
		we only need to prove that the last summands in the splitting are simple.
		By \cite[Theorem 5.2.5]{Dimca2}, the simplicity of an intermediary extension in the category of perverse sheaves amounts to the irreducibility of the corresponding local system.

		We start with the simplicity of $ \iota_*\Sigma$.
		We observe that that fundamental group  $\pi_1(\vc, H)$ acts on its generators $ \sigma_1^{\vee}, \dots \sigma_{16}^{\vee}$ by  ``exchanging the branches''  of the normal crossings divisor $\vX\cap T$. So, in order to conclude the proof, it suffices to prove that 
		the monodromy acts transitively on the nodes.
		In order to do this, fix a Kummer quartic $S\in \vc$ and denote by $P_1, \dots , P_{16}$ the nodes of $S$. By letting $S$ to vary in $\vc$, the points $P_1, \dots , P_{16}$ span a subvariety $\nc \subset \hc \mid_{\vc}$. The natural map $\kappa:\nc \rightarrow \vc$ is an \'etale morphism of degree $16$. In order to prove the statement, it suffices to show that $\nc $ is connected. By \cite[Theorem 10.3.18]{Dolgachev}, any Kummer quartic is projectively isomorphic to a quartic surface described by a polynomial varying in the  Segre cubic primal $\Sc \subset \vc$, which is the dual variety of the Castelnuovo-Richmond quartic $CR_4$. In order to prove our claim, it is enough to prove that 
		$\kappa^{-1}(\Sc)$ is connected, so from now on we assume $S\in \Sc$. But the connectedness of $\kappa^{-1}(\Sc)$ simply follows from the description of $CR_4$, given in \cite[p. 608]{Dolgachev}, as the moduli space of abelian surfaces with full level two structures. Indeed,  moving around in $\Sc$ we are allowed to change as we like  the level two structure of the abelian surface associated to $S$. This amounts to a transitive action of the monodromy group on the set $P_1, \dots , P_{16}$ because the nodes of any surface in $\Sc $ form an orbit of the Heisenberg group \cite[p. 608]{Dolgachev}.
		
		As for the irreducibility  of $ \overline{ \Sigma}^{\perp}_{prim}$, since $ \pi_1(T, H)\cong \pi_1(\vc, H) $ ($H\in \vc$), we are left with the prove of the irreducibility of   $\overline{\Sigma} ^{\perp}_{prim} \mid _{\vc}$.
		
		To this end we argue as in Section 1. Consider $\ac= \ac _{(2,2)}$ the \textit{analytic moduli space of abelian surfaces with polarization} $(2,2)$, i.e. the quotient 
		$$\hc_2/G_{(2,2)},$$
		where $\hc_2$ is the Siegel upper half space and where the discrete group $G_{(2,2)}$ is defined in \cite[\S 8]{BL}. The moduli space $\ac$ is a coarse moduli space for the functor parametrizing isomorphisms classes of abelian surfaces of type $(2,2)$ \cite[Corollary 8.2.7]{BL}.
		
		Set
		$$\tau: M \rightarrow \ac$$
		the universal family over $\ac$ and by $\lc \rightarrow M$ the line bundle on $M$ providing the polarization $(2,2)$.
		The direct image $\tau_*(\lc)$ is a vector bundle of rank $4$ on $\ac$.
		As in Section 2, consider the projective bundle 
		$$\xi:\ec:=\bP(Fr(\tau_*(\lc)))\rightarrow \ac,$$
		where $Fr(\tau_*(\lc))$  denotes the frame bundle of the  direct image $\tau_*(\lc).$

		A point in $\ec$ is a pair $(F,A)$, where $A\in \ac$ is an abelian surface and $F$ is a frame of the linear system
		$\mid \lc_A \mid ,$
		thus representing a map
		$A\rightarrow \bP^3,$
		whose image is a Kummer quartic in $\bP^3$.
		
		By letting $A$ to vary in $\ac$, we get a morphism
		\begin{equation}\label{univfam}
			\phi: M\times _{\ac} \ec\rightarrow \bP^3 \times \ec,	
		\end{equation}
		whose image is the restriction to $\vc$ of  $\mathcal H\subset \bP^3 \times \bP$, the universal family of quartic surfaces of $\bP^3$.
		By the universal property of Hilbert scheme, the family (\ref{univfam}) can be recovered from $\hc$  via base change with a morphism 
		\begin{equation}
			\label{psi}
			\psi: \,\, \ec \longrightarrow	\vc.
		\end{equation}
		To prove the irreducibility of $\overline{\Sigma} ^{\perp}_{prim} \mid _{\vc}$ it suffices to observe that its pull-back 	 $\psi^*\overline{\Sigma} ^{\perp}_{prim} \mid _{\vc}$ is irreducible since it coincides with the primitive part of
		$\xi^* R^2\tau{_*}\mathbb Q _M. $

	Indeed, the irreducibility of such a  local system is a consequence of the irreducibility of the action of the symplctic group on the primitive cohomology of a polarized abelian surface, which follows from the representation theory of  symplectic group (compare with \cite[par. 10.2]{FH}), and which is the main ingredient of the proof of the Hodge conjecture for the general polarized abelian variety (compare with \cite[Ch. 17]{BL}).

	\end{proof}

\end{document}